\documentclass[11pt, a4paper]{amsart}

\usepackage{a4, a4wide}
\usepackage{amssymb}
\usepackage{amsmath}
\usepackage{amsthm}
\usepackage{amstext}
\usepackage{amscd}
\usepackage{enumitem}
\usepackage{latexsym}
\usepackage{graphics}
\usepackage{color}
\usepackage[all]{xy}

\usepackage[colorlinks, pagebackref]{hyperref}
\hypersetup{
  colorlinks=true,
  citecolor=blue,
  linkcolor=blue,
  urlcolor=blue}

\makeatletter
\def\@tocline#1#2#3#4#5#6#7{\relax
  \ifnum #1>\c@tocdepth 
  \else
    \par \addpenalty\@secpenalty\addvspace{#2}%
    \begingroup \hyphenpenalty\@M
    \@ifempty{#4}{%
      \@tempdima\csname r@tocindent\number#1\endcsname\relax
    }{%
      \@tempdima#4\relax
    }%
    \parindent\z@ \leftskip#3\relax \advance\leftskip\@tempdima\relax
    \rightskip\@pnumwidth plus4em \parfillskip-\@pnumwidth
    #5\leavevmode\hskip-\@tempdima
      \ifcase #1
       \or\or \hskip 2em \or \hskip 2em \else \hskip 3em \fi%
      #6\nobreak\relax
    \dotfill\hbox to\@pnumwidth{\@tocpagenum{#7}}\par
    \nobreak
    \endgroup
  \fi}
\makeatother

\newtheorem{intro-thm}{Theorem}[]

\theoremstyle{plain}
\newtheorem{theorem}{Theorem}[section]

\newtheorem{lemma}[theorem]{Lemma}

\theoremstyle{definition}
\newtheorem{remark}[theorem]{Remark}

\newtheorem{definition}[theorem]{Definition}

\newtheorem*{claim*}{Claim}

\numberwithin{equation}{section}

\def\~{\widetilde}
\def\-{\overline}
\def\<{\langle}
\def\>{\rangle}
\def\@{\mathcal}
\def\!{\mathscr}
\def\#{\mathbb}
\def\<{\langle}
\def\>{\rangle} 
\def\-{\overline} 
\def\~{\widetilde}
\def\^{\widehat}

\newcommand{\Spec}{{\rm Spec \,}}
\newcommand{\Sing}{{\rm Sing}_*^{\mathbb A^1}}

\input{xy}
\xyoption{all}

\begin{document}
\title{Remarks on iterations of the $\#A^1$-chain connected components construction}

\author{Chetan Balwe}
\address{Department of Mathematical Sciences, Indian Institute of Science Education and Research Mohali, Knowledge City, Sector-81, Mohali 140306, India.}
\email{cbalwe@iisermohali.ac.in}

\author{Bandna Rani}
\address{Department of Mathematical Sciences, Indian Institute of Science Education and Research Mohali, Knowledge City, Sector-81, Mohali 140306, India.}
\email{bandnarani@iisermohali.ac.in}

\author{Anand Sawant}
\address{School of Mathematics, Tata Institute of Fundamental Research, Homi Bhabha Road, Colaba, Mumbai 400005, India.}
\email{asawant@math.tifr.res.in}
\date{}

\thanks{Chetan Balwe acknowledges the support of SERB-DST MATRICS Grant: MTR/2017/000690}
\thanks{Anand Sawant acknowledges the support of SERB Start-up Research Grant SRG/2020/000237 and the Department of Atomic Energy, Government of India, under project no. 12-R\&D-TFR-5.01-0500.}


\begin{abstract} 
We show that the sheaf of $\#A^1$-connected components of a Nisnevich sheaf of sets and its universal $\#A^1$-invariant quotient (obtained by iterating the $\#A^1$-chain connected components construction and taking the direct limit) agree on field-valued points.  This establishes an explicit formula for the field-valued points of the sheaf of $\#A^1$-connected components of any space.  Given any natural number $n$, we construct an $\#A^1$-connected space on which the iterations of the naive $\#A^1$-connected components construction do not stabilize before the $n$th stage.
\end{abstract}

\maketitle

\setlength{\parskip}{2pt plus2pt minus1pt}
\raggedbottom

\section{Introduction}
\label{section introduction}
Fix a base field $k$ and let $Sm_k$ denote the category of smooth, finite type, separated schemes over $k$.  For any Nisnevich sheaf of sets $\@F$ on $Sm_k$, the $\#A^1$-chain connected components construction produces a sheaf $\@S(\@F)$, which is obtained by taking the Nisnevich sheafification of the presheaf that associates with every smooth scheme $U$ over $k$, the quotient of the set $\@F(U)$ by the equivalence relation generated by \emph{naive $\#A^1$-homotopies}; that is, elements of $\@F(U \times \#A^1)$ (see \cite[Definition 2.9]{Balwe-Hogadi-Sawant}).  Iterating this construction infinitely many times and taking the direct limit gives the \emph{universal $\#A^1$-invariant quotient} $\@L(\@F)$ of $\@F$ (see \cite[Theorem 2.13, Remark
2.15, Corollary 2.18]{Balwe-Hogadi-Sawant} or \cite[Theorem 2.5]{Sawant-IC}).  The universal $\#A^1$-invariant quotient of a sheaf $\@F$ agrees with the sheaf $\pi_0^{\#A^1}(\@F)$ of its $\#A^1$-connected components, provided the latter sheaf is $\#A^1$-invariant.  The $\#A^1$-invariance of the sheaf of $\#A^1$-connected components of a \emph{space} (that is, a simplicial Nisnevich sheaf of sets on $Sm_k$) was conjectured in \cite[Conjecture 1.12]{Morel}.  In view of \cite[Proposition 2.17]{Balwe-Hogadi-Sawant}, one then obtains an explicit conjectural formula for the sheaf of $\#A^1$-connected components of any space $\@X$ in terms of iterations of the $\#A^1$-chain connected components construction on the simplicial connected components $\pi_0(\@X)$ of $\@X$: 
\begin{equation}
\label{eqn:iterations}
\pi_0 ^{\#A^1}(\@X) = \@L(\pi_0(\@X)) = \varinjlim_n~ \@S^n(\pi_0(\@X)).
\end{equation}

The first main result of this short note (Theorem \ref{theorem field valued points}) verifies the conjectural formula \eqref{eqn:iterations} over sections over finitely generated separable field extensions of the base field.  Since any space is $\#A^1$-weakly equivalent to a sheaf of sets by \cite[Lemma 3.12, page 90]{Morel-Voevodsky}, Theorem \ref{theorem field valued points} gives an explicit formula for the sections of $\#A^1$-connected components of any space over fields.

\begin{theorem}[See Theorem \ref{theorem field valued points}]
Let $\@F$ be a sheaf of sets on $Sm/k$. For any finitely generated, separable field extension $K/k$, the natural map $\pi_0^{\#A^1}(\@F)(\Spec K) \to \@L(\@F)(\Spec K)$ is a bijection. 
\end{theorem}

The second main result of this short note shows that the iterations of the functor $\@S$ appearing in \eqref{eqn:iterations} are indeed necessary.  

\begin{theorem}
\label{main theorem}
There exists a sequence $\@X_n$ of $\#A^1$-connected spaces (which are, in fact, sheaves of sets) such that $\@S^{n+1}(\@X_n) = \pi_0^{\#A^1}(\@X_n)$ is the trivial one-point sheaf, but $\@S^i(\@X_n) \neq \@S^{i+1}(\@X_n)$, for every $i < n+1$.
\end{theorem}

We construct this sequence of spaces inductively by starting with $\@X_0 = \#A^1_k$, considering an elementary Nisnevich cover of it and then successively considering certain pushouts of \emph{total spaces} of a \emph{ghost homotopy} defined over such a cover in the sense of \cite[Section 2.2]{Balwe-Sawant-ruled}.  The construction, described in Section \ref{section construction}, is clearly inspired by \cite[Construction 4.3]{Balwe-Hogadi-Sawant}, but is much simpler and more elegant. 

\subsection*{Acknowledgements}
We thank Fabien Morel for his comments on this article.

\section{Field-valued points of \texorpdfstring{$\@L(\@X)$}{L(X)}}

For any sheaf of sets $\@F$ on $Sm/k$, we will write $\@F(A)$ for $\@F(\Spec A)$ for any affine $k$-algebra $A$ for brevity.  We will follow the notation and conventions of \cite{Balwe-Hogadi-Sawant} and \cite{Balwe-Sawant-ruled}.

If $X$ is a proper scheme over $k$, it is known that the map $\@S(X)(K) \to \pi_0^{\#A^1}(X)(K)$ is a bijection for any finitely generated field extension $K/k$ (see \cite[Corollary 2.2.6]{Asok-Morel} and \cite[Corollary 3.10]{Balwe-Hogadi-Sawant}). Stated in this form, this result can be easily seen to fail for non-proper schemes. However, for a proper scheme, we also know that $\@S(X)(K) \to \@L(X)(K)$ is a bijection for any finitely generated field extension $K/k$ (see \cite[Theorem 3.9]{Balwe-Hogadi-Sawant}). Thus, the above result can be rephrased by saying that for a proper scheme $X$, the function $\pi_0^{\#A^1}(X)(K) \to \@L(X)(K)$ is a bijection for any finitely generated field extension $K/k$. We will prove that this statement holds when the proper scheme $X$ is replaced by any sheaf of sets. 

For a smooth scheme $X$, an \emph{elementary Nisnevich cover} of $X$ is a pair of morphisms $(\pi: V_1 \to X, \pi_2: V_2 \to X)$ such that $V_1 \to X$ is an open immersion and $\pi_2^{-1}(X \backslash V_1) \to X \backslash V_1$ is an isomorphism (where we use the reduced structure on $X \backslash V_1$). 

\begin{lemma}
\label{lemma trim cover}
Let $C$ be a smooth irreducible curve over $k$. Let $(V_1 \to C, V_2 \to C)$ be an elementary Nisnevich cover of $C$. Let $W \to V_1 \times_C V_2$ be a Nisnevich cover. Then, there exists an open subscheme $V'_2 \subset V_2$ such that $(V_1 \to C, V'_2 \to C)$ is an elementary Nisnevich cover of $C$ and the morphism 
\[
(V_1 \times_C V'_2) \times_{(V_1 \times_C V_2)} W \to V_1 \times_C V'_2
\] 
is an isomorphism. 
\end{lemma}

\begin{proof}
As $C$ is irreducible, $V_1 \to C$ is a dense open immersion. Thus, $V_1 \times_C V_2 \to V_2$ is a dense immersion. As $C$ is $1$-dimensional, $C \backslash V_1$ is a finite set $Z$, each element of which is a closed point. 

There exists a scheme $W'$, which is a union of some components of $W$ such that the morphism $W' \hookrightarrow W \to V_1 \times_C V_2$ is a dense open immersion. Thus, $W' \to V_2$ is a dense open immersion. We define $V'_2 = W' \cup Z$. We see that $(V_1 \to C, V'_2 \to C)$ is an elementary Nisnevich cover of $C$. Since $V'_2 \times_C V_2 = W'$, we see that
\[
(V_1 \times_C V'_2) \times_{(V_1 \times_C V_2)} W \to V_1 \times_C V'_2
\] 
is an isomorphism. 
\end{proof}

\begin{theorem}
\label{theorem field valued points}
Let $\@F$ be a sheaf of sets. For any finitely generated field extension $K/k$, the natural map $\pi_0^{\#A^1}(\@F)(K) \to \@L(\@F)(K)$ is a bijection. 
\end{theorem}

\begin{proof}
We need to prove that if $x_1, x_2 \in \@F(K)$ are connected by an $n$-ghost homotopy (in the sense of \cite[Definition 2.7]{Balwe-Sawant-ruled}), they map to the same element of $\pi_0^{\#A^1}(\@F)(K)$. We prove this by induction on $n$. The case $n = 0$ is trivial. So we now assume that the claim holds for points that are connected by a $m$-ghost homotopy for $m < n$. 

Let $\@X = L_{\#A^1}(\@F)$. Let 
\[
\@H = (V \to \#A^1_K, W \to V \times_{\#A^1_K} V, h, h^W)
\] 
be an $n$-ghost homotopy connecting $x_1$ and $x_2$ in $\@F$. We may write $V= V_1 \cup V_2$ where $V_1 \to \#A^1_K$ is an open immersion. By shrinking $V_2$ if necessary, we may assume that the morphism $V_2 \to \#A^1_K$ is an isomorphism on the complement of $V_1$, i.e. $(V_1 \to \#A^1_K, V_2 \to \#A^1_K)$ is an elementary Nisnevich cover of $\#A^1_K$. Let $h_i = h|_{V_i}$ for $ i = 1,2$. Using Lemma \ref{lemma trim cover}, we may shrink $V_2$ if necessary and assume that we actually have an $(n-1)$-ghost homotopy $\@H'$ connecting $h_1|_{V_1 \times_{\#A^1_F} V_2}$ to $h_2|_{V_1 \times_{\#A^1_F} V_2}$. For the sake of convenience, we denote $V_1 \times_{\#A^1_F} V_2$ by $W'$. Observe that $W' \to V_2$ is an open immersion. 

Let $W' = \bigcup_{i=1}^p W'_i$ be the decomposition of $W'$ into irreducible components. Let $\eta_i: \Spec L_i \to W_i$ be the generic point of $W_i$. Then $\@H'_{\eta_i}$ is an $(n-1)$-ghost homotopy of $\Spec L_i$ in $X$. By the induction hypothesis it follows that the morphisms $h_i|_{\eta_i}$, $i=1,2$, are simplicially homotopic in $\@X(L_i)$. This simplicial homotopy extends to an open subset $W''_i$ of $W'_i$. Thus, if $W'' = \bigcup W''_i$, then we see that $h_i|_{W''}$, $i = 1,2$ are simplicially homotopic in $\@X$. By further shrinking $V_2$, we may now assume that $W'' = V_1 \times_{\#A^1_F} V_2$. 

Since $W'' \rightarrow V_2$ is an open immersion and since $\@X$ is simplicially fibrant, this simplicial homotopy extends to a simplicial homotopy of $V_2$. Suppose this simplicial homotopy connects $h_2$ to $h'_2: V_2 \to \@X$. Then $h_1$ and $h'_2$ can be glued together to give a morphism $\#A^1_K$ to $\@X$. This shows that $x_1$ and $x_2$ are simplicially homotopic in $\@X$. Thus, they map to the same point of $\pi_0^{\#A^1}(\@F)(K)$. 
\end{proof}

\section{Closed embeddings of sheaves}
\label{section preliminaries}

\begin{definition}
Let $\@F$ and $\@G$ be Nisnevich sheaves of sets on $Sm_k$ and let $i: \@F \to \@G$ be a monomorphism.  We say that $i$ is a \emph{closed embedding of sheaves} if it has the right lifting property with respect to any dense open immersion $U \hookrightarrow X$, where $X$ is a smooth variety over $k$.  
\end{definition}

Observe that if $X$ is a smooth variety over $k$ and $\eta: \Spec K \to X$ is the inclusion of the generic point of $X$, then any closed embedding of sheaves has the right lifting property with respect to $\eta$.  The analogue of this fact for Nisnevich sheaves is as follows.

\begin{lemma}
\label{lemma Henselian criterion for closed}
A monomorphism $i: \@F \to \@G$ is a closed embedding of sheaves if and only if for any smooth henselian local scheme $X$ with generic point $\eta: \Spec K \to X$, the morphism $i$ has the right lifting property with respect to $\eta$.  
\end{lemma}

\begin{proof}
Let $X$ be smooth variety over $k$ and let $U$ be an open subset of $X$. Suppose that we have a diagram
\[
\xymatrix{
U \ar[r] \ar[d] & \@F \ar[d] \\
X \ar[r] & \@G \text{.}
}
\]
For any point $x$ of $X$, let $X_x$ denote the scheme $\Spec \@O_{X,x}^h$ and let $\eta_x: \Spec K_x \to X_x$ denote the generic point of $X_x$. Then, in the diagram 
\[
\xymatrix{
\Spec K_x \ar[r] \ar[d]_{\eta_x} & U \ar[r] \ar[d] & \@F \ar[d] \\
X_x \ar[r] \ar@{-->}[urr] & X \ar[r] & \@G \text{.}
}
\]
we obtain a morphism $X_x \to \@F$ (indicated by the dashed arrow) making the diagram commute. This means that there exists a smooth variety $X'_x$ with an \'etale map $\pi: X'_x \to X$, which is an isomorphism on $\pi^{-1}(x)$ and such that we have a morphism $X'_x \to \@F$ making the diagram 
\[
\xymatrix{
 & U \ar[r] \ar[d] & \@F \ar[d] \\
X'_x \ar[r] \ar@{-->}[urr] & X \ar[r] & \@G 
}
\]
commute.  Thus, we see that there exists a Nisnevich cover $X' \to X$ and a morphism $X' \to \@F$ such that the diagram  
\[
\xymatrix{
 & U \ar[r] \ar[d] & \@F \ar[d] \\
X'\ar[r] \ar@{-->}[urr] & X \ar[r] & \@G 
}
\] 
commutes. The morphism $X' \to \@F \to \@G$ descends to a morphism $X \to \@G$.  Since $\@F \to \@G$ is a monomorphism, we see that the morphism $X' \to \@F$ also descends to a morphism $X \to \@F$ making the lower triangle in the diagram
\[
\xymatrix{
U \ar[r] \ar[d] & \@F \ar[d] \\
X \ar[r] \ar[ur] & \@G \text{.}
}
\]
commute, which in turn, makes the upper triangle commute. 
\end{proof}

\begin{lemma}
\label{lemma closed embedding} Let $\@F$, $\@G$ and $\@H$ be Nisnevich sheaves of sets on $Sm_k$.
\begin{enumerate}[label=$(\alph*)$]
\item If $\@F \to \@G$ is a closed embedding of sheaves, then for any morphism $\@H \to \@G$, the morphism $\@F \times_{\@G} \@H \to \@G$ is a closed embedding of sheaves. 
\item Let $p:\@F \to \@G$ be an epimorphism and let $i:\@H \to \@G$ be a monomorphism of sheaves. If $i': \@F \times_{\@G} \@H \to \@F$ is a closed embedding, then so is $i$.
\end{enumerate}
\end{lemma}
\begin{proof}
Part $(a)$ is obvious; we prove part $(b)$.  Let $p': \@F \times_{\@G} \@H \to \@H$ be the projection on the second factor. Let $X$ be a smooth henselian local scheme with generic point $\eta: \Spec K \to X$. Let $\alpha: X \to \@G$ such that $\alpha \circ \eta$ factors through $i$. As $p$ is an epimorphism, there exists $\beta: X \to \@F$ such that $p \circ \beta = \alpha$.  Since $p \circ \beta \circ \eta = \alpha \circ \eta$ factors through $i$, the morphism $\beta \circ \eta$ factors through $i'$. As $i'$ is a closed embedding, $\beta$ factors through $i'$.  Thus, $\beta = i' \circ \beta'$ and we have
\[
\alpha = p \circ \beta = p \circ i' \circ \beta' = i \circ p' \circ \beta'.
\]
Hence, $\alpha$ factors through $i$. This proves that $i$ is a closed embedding. 

\end{proof}

\section{Proof of Theorem \ref{main theorem}}
\label{section construction}

Consider the Zariski cover of $\#A^1_k$ given by $V_1 = \#A^1_k \setminus \{1\}$ and $V_2 = \#A^1_k \setminus \{0\}$. Let $p_1: V_1 \to \#A^1_k$ and $p_2: V_2 \to \#A^1_k$ be the inclusion morphisms. Let $W := V_1 \times_{\#A^1_k} V_2 = \#A^1 \setminus \{0,1\}$. For $i = 1,2$, let $\pi_i: V_1 \times_{\#A^1_k} V_2 \to V_i$ be the projection (which is an open immersion). 

We will now inductively construct a sequence of sheaves $\{\@X_n\}_{n \in \#Z_{\geq -1}}$ on $Sm_k$ and morphisms $\alpha_n, \beta_n: \Spec k \to \@X_n$.  Set $\@X_{-1} := \Spec k$ and let $\alpha_{-1}, \beta_{-1}: \Spec k \to \Spec k$ be the identity maps. 

If $\@X_{n-1}, \alpha_{n-1} , \beta_{n-1}$ are defined, we define $\@X_n$ to be the pushout of the diagram 
\[
\xymatrix{
W \coprod W \ar[r]^{\phi_n} \ar[d]_{\psi_n} & V_1 \coprod V_2 \ar[d]^{\psi'_n} \\
W \times \@X_{n-1} \ar[r]_{\phi'_n} &  \@X_n 
}
\] 
where $\phi_n = \pi_1 \coprod \pi_2$ and $\psi_n = id_W \times (\alpha_{n-1} \coprod \beta_{n-1})$. We define $\alpha_n: \Spec k \to \@X_n$ to be the composition of $\Spec k \stackrel{0}{\to} V_1 \to \@X_n$ and $\beta_n: \Spec k \to \@X_n$ to be the composition of $\Spec k \stackrel{1}{\to} V_2 \to \@X_n$. Clearly, $\@X_0 = \#A^1$ and $\alpha_0, \beta_0$ are the morphisms $\Spec k \to \#A^1$ corresponding to the points $0$ and $1$. 

\begin{lemma}
\label{lemma marked points closed embedding}
The morphism $\alpha_n \coprod \beta_n: \Spec k \coprod \Spec k \to \@X_n$ is a closed embedding. 
\end{lemma}

\begin{proof}
Let $\@P = \Spec k \coprod \Spec k$ and let $\gamma: \@P \to \@X_n$ denote the morphism $\alpha_n \coprod \beta_n$. Let $\@Q := \@P \times_{\gamma_n, \@X_n, \psi'_n} (V_1 \coprod V_2)$, and let $pr_1$ and $pr_2$ denote the projections of this fiber product to the first and second factors respectively. As $\psi'_n$ is a monomorphism, the projection $pr_1$ is also a monomorphism. 

The composition $\@P \stackrel{0 \coprod 1}{ \to } V_1 \coprod V_2$ and the identity morphism $id_{\@P}: \@P \to \@P$ induce a morphism $\@P \to \@Q$ such that the composition $\@P \to \@Q \stackrel{pr_1}{\to} \@P$ is equal to $id_{\@P}$. Thus, we see that $pr_1$ is an epimorphism. Thus, $pr_1$ is an isomorphism.

If we identify $\@Q$ with $\@P$ using $pr_1$, then $pr_2$ may be identified with the morphism $0 \coprod 1$, which is a closed embedding of $\@P$ into $V_1 \coprod V_2$. Thus, by Lemma \ref{lemma closed embedding} $(b)$, we see that $\gamma$ is a closed embedding. 
\end{proof}

In what follows, the following simple observation will be useful.

\begin{lemma}
\label{lemma factoring morphisms through subsheaf}
Let $f: \@X \to \@Y$ be a monomorphism of Nisnevich sheaves. Let $\tau$ be a Grothendieck topology on $Sm_k$ which is finer than the Nisnevich topology. Suppose that $\@X$ is a sheaf for $\tau$. Then, $f$ has the right lifting property with respect to any $\tau$-cover $V \to U$ where $U$ is an essentially smooth scheme over $k$.  
\end{lemma}

\begin{proof}
Suppose we have a diagram
\[
\xymatrix{
V \ar[r]^{\alpha} \ar[d] & \@X \ar[d]^{f} \\
U \ar[r]_{\beta} & \@Y \text{.} 
}
\]
Since $\@X$ is a sheaf for the topology $\tau$, it suffices to prove that the two morphisms
\[
V \times_U V \rightrightarrows V \stackrel{\alpha}{\to} \@X
\]
are equal. However, we see that the compositions of these morphisms with $f$ are equal since $f \circ \alpha$ factors through $\beta$. Thus, the result follows since $f$ is a monomorphism. 
\end{proof}

Define $\@Y_n := W \times \@X_{n-1} \coprod V_1 \coprod V_2$ and let $p_n: \@Y_n \to \@X_n$ be the morphism induced by $\phi'_n: W \times \@X_{n-1} \to \@X_n$ and $\psi'_n: V_1 \coprod V_2 \to \@X_n$. Let $\phi''_n: W \times \@X_n \to \@Y_n$ and $\psi''_n: V_1 \coprod V_2 \to \@Y_n$ be the obvious (inclusion) maps. Thus, $\phi'_n = p_n \circ \phi''_n$ and $\psi'_n = p_n \circ \psi''_n$. 

\begin{lemma}
\label{lemma lifting morphism to Y_n}
Let $n \geq 1$ be an integer. For any essentially smooth irreducible scheme $Z$, any morphism $\alpha: Z \to \@X_n$ factors through $p_n$. 
\end{lemma}

\begin{proof}
Clearly, $p_n$ is an epiomorphism of sheaves. Thus, for any scheme $Z$ and any morphism $\alpha: Z \to \@X_n$, there exists a Nisnevich cover $\{\gamma_i: Z_i \to Z \}_{i \in I}$ such that for each $i$, the the morphism $\alpha|_{Z_i}:= \alpha \circ \gamma_i$ is equal to $p_n \circ\beta_i$ for some $\beta_i: Z_i \to \@Y_n$. Let us assume that each $Z_i$ is irreducible and let $\eta_i: \Spec K_i \to Z_i$ be the generic points. Note that $I$ may be taken to be a finite set. 

Each $\beta_i$ factors through $\phi''_n$ or $\psi''_n$. Suppose that all the $\beta_i$ factor through $\phi''_n$. Then, as $\phi'_n$ is a monomorphism, by Lemma \ref{lemma factoring morphisms through subsheaf} we see that there exists a morphism $\beta: Z \to W \times \@X_{n-1}$ such that $\alpha = \phi'_n \circ \beta$. Similarly, if all the $\beta_i$ factor through $\psi''_n$, then there exists a morphism $\beta: Z \to V_1 \coprod V_2$ such that $\alpha = \psi'_n \circ \beta$. We claim that neither of these conditions hold, then we can change some of the $\beta_i$'s to reduce to the situation where they all factor through $\psi''_n$. 

Thus, now let assume that we can find two indices $i,j \in I$ such that $\beta_i$ factors through $\psi''_n$ and $\beta_j$ factors through $\phi''_n$. Thus, there exists a morphism $\beta'_i: Z_i \to V_1 \coprod V_2$ such that $\beta_i = \psi''_n \circ \beta'_i$ and a morphism $\beta'_j : Z_j \to W \times \@X_{n-1}$ such that $\beta_j = \phi''_n \circ \beta'_j$. 

Let $P$ be a component of $Z_i \times_{Z} Z_j$ and let $\Spec L \to P$ be the generic point. Let $\rho_i$ denote the composition
\[
\Spec L \to \Spec K_i\stackrel{\eta_i}{\to} Z_i \stackrel{\beta'_i}{\to} V_1 \coprod V_2 \stackrel{\psi''_n}{\to} \@Y_n 
\]
and let $\rho_j$ denote the composition
\[
\Spec L \to \Spec K_i \stackrel{\eta_j}{\to} Z_j \stackrel{\beta'_j}{\to} W \times \@X_{n-1} \stackrel{\phi''_n}{\to} \@Y_n \text{.}
\]
Then, $p_n \circ \rho_i = p_n \circ \rho_j$. Thus, $\rho_i$ factors through $\phi_n: W \coprod W \to V_1 \coprod V_2$ and $\rho_j$ factors through $\psi_n: W \coprod W \to W \times \@X_{n-1}$.

Since $L/K_j$ is a separable field extensions and since $W \coprod W$ is a scheme (and hence, an \'etale sheaf), by Lemma \ref{lemma factoring morphisms through subsheaf}, the morphism $\beta'_j \circ \eta_j$ factors through $\psi_n$. By Lemma \ref{lemma closed embedding}$(a)$ and Lemma \ref{lemma factoring morphisms through subsheaf}, $\psi_n$ is a closed embedding.  Thus, we see that $\beta'_j$ factors through $\psi_n$. Let $\beta''_j: Z_j \to W \coprod W$ be such that $\beta'_j =  \psi_n \circ \beta''_j$. Let $\~{\beta}_j: Z_i \to \@Y_n$ be the composition
\[
Z_j \stackrel{\beta''_j}{\to} W \coprod W \stackrel{\phi_n}{\to} V_1 \coprod V_2 \to \@Y_n \text{.}
\] 
Observe that $p_n \circ \~{\beta}_j = p_n \circ \beta_j$. Thus, we may now replace $\beta_j$ by $\~{\beta}_j$. 

We now repeat this process until we come to a situation where all the $\beta_i$'s factor through $\psi''_n$. This completes the proof. 
\end{proof}

Let $\Sing$ denote the Morel-Voevodsky singular construction \cite[page 87]{Morel-Voevodsky}.

\begin{lemma}
Let $n \geq 1$ be an integer. The square
\[
\xymatrix{
W \coprod W \ar[r] \ar[d] & V_1 \coprod V_2 \ar[d] \\
W \times \Sing \@X_{n-1} \ar[r] & \Sing\@X_n 
}
\]
is a pushout square. 
\end{lemma}

\begin{proof}
We will prove that for any essentially smooth irreducible scheme $Z$ over $k$, the square 
\[
\xymatrix{
(W \coprod W)(Z) \ar[d]_{(\psi_n)_Z} \ar[r]^{(\phi_n)_Z} & (V_1 \coprod V_2)(Z) \ar[d]^{(\psi'_n)(Z)} \\
(W \times \@X_{n-1})(Z) \ar[r]_{(\phi'_n)_Z} & \@X_n(Z)
}
\]
is a commutative square. Once this is proved, we take $Z = U \times \#A^m$ for $m \geq 0$ where $U$ is smooth henselian local scheme over $k$. Now, using the fact that $V_1$, $V_2$ and $W$ are $\#A^1$-rigid, the result follows.  

By Lemma \ref{lemma lifting morphism to Y_n}, the function
\[
(W \coprod \@X_{n-1})(Z) \coprod (W \coprod W)(Z) \to \@X_n(Z)
\]
is a surjection. Also, the functions $(\phi_n)_Z$, $(\psi_n)_Z$, $*(\phi'_n)_Z$ and $(\psi'_n)_Z$ are injective. Thus, it suffices to show that if $\alpha \in (V_1 \coprod V_2)(Z)$ and $\beta \in (W \times \@X_{n-1})(Z)$ are such that $(\psi'_n)_Z(\alpha) = (\phi'_n)_Z(\beta)$, then there exists $\gamma \in (W \coprod W)(Z)$ such that $(\phi_n)_Z(\gamma) = \alpha$ and $(\psi_n)_Z(\gamma) = \beta$. 

If $(\psi'_n)_Z(\alpha) = (\phi'_n)_Z(\beta)$, there exists a Nisnevich cover $Z' \to Z$ such that there exists $\gamma' \in (W \coprod W)(Z')$ such that $(\phi_n)_{Z'}(\gamma') = \alpha|_{Z'}$ and $(\psi_n)_{Z'}(\gamma') = \beta|_{Z'}$. By Lemma \ref{lemma factoring morphisms through subsheaf}, $\gamma'$ factors through a morphism $\gamma: Z \to W \coprod W$. This completes the proof. 
\end{proof}

\begin{theorem}
\label{theorem construction}
Let $n \geq 0$ be an integer. 
\begin{enumerate}[label=$(\arabic*)$]
\item $\Sing \@X_n$ is simplicially equivalent to  $\@X_{n-1}$. 
\item $\@S(\@X_n) \cong \@X_{n-1}$.   
\end{enumerate}
\end{theorem}

\begin{proof}
We prove this theorem by induction on $n$. Since $\@X_0 = \#A^1$, the result is easily seen to be true for $n = 0$. 

Suppose the result is known to be true for $n \leq m$ where $m \geq 0$. In the diagram
\[
\xymatrix{
W \times \Sing \@X_{m} \ar[d] & W \coprod W \ar[r] \ar[l] \ar[d] & V_1 \coprod V_2 \ar[d] \\
W \times \@X_{m-1}  & W \coprod W \ar[r] \ar[l] & V_1 \coprod V_2
}
\]
the vertical arrows are simplicial equivalences and the horizontal arrows are cofibrations. Thus, we see that the pushouts of these diagrams are simplicially equivalent. Thus, $\Sing \@X_n$ is simplicial equivalent to $\@X_{n-1}$. This proves (1).  Since $\@S(\@X_n) = \pi_0 (\Sing \@X_n)$ by definition, a similar argument proves (2). 
\end{proof}

\subsection*{Proof of Theorem \ref{main theorem}}

A repeated application of Theorem \ref{theorem construction} shows that $\@S^{n+1}(\@X_n)$ is the trivial one point sheaf.  It also shows that $\@X_n$ is $\#A^1$-connected.  It is clear from the construction of $\@X_n$ that $\@S^i(\@X_n) \neq \@S^{i+1}(\@X_n)$, for every $i < n+1$.

\begin{remark}
Note that $\#A^1$-connectedness of the spaces $\@X_n$ may also be obtained directly by a repeated application of \cite[Lemma 4.1]{Balwe-Hogadi-Sawant}.
\end{remark}

\begin{remark}
It may be possible to modify the construction described above by starting with the Nisnevich cover of $\#A^1$ used in \cite[Construction 4.3]{Balwe-Hogadi-Sawant} and \cite[Construction 4.5]{Balwe-Hogadi-Sawant} to produce a sequence of schemes $X_n$ over $k$ such that $\@S^n(X_n) \neq \@S^{n+1}(X_{n})$.
\end{remark}

\end{document}